\documentclass[11pt]{article}

\usepackage[T1]{fontenc}
\usepackage{lmodern}
\usepackage{microtype}
\usepackage{amsmath,amssymb,amsthm,mathtools}
\usepackage{geometry}
\usepackage{booktabs}
\usepackage{enumitem}
\usepackage{multicol}
\usepackage{float}
\usepackage{tikz}
\usepackage{listings}
\usepackage[hidelinks]{hyperref}

\allowdisplaybreaks
\newtheorem{theorem}{Theorem}[section]
\newtheorem{lemma}[theorem]{Lemma}
\newtheorem{proposition}[theorem]{Proposition}

\newtheorem{definition}[theorem]{Definition}
\newtheorem{remark}[theorem]{Remark}
\newtheorem{example}[theorem]{Example}
\newtheorem{observation}[theorem]{Observation}

\newcommand{\Z}{\mathbb Z}
\newcommand{\Zi}{\mathbb Z[i]}
\newcommand{\U}{\mathrm U}
\newcommand{\V}{\mathrm V}
\newcommand{\cC}{\mathcal C}
\newcommand{\cR}{\mathcal R}
\newcommand{\abs}[1]{\lvert#1\rvert}

\numberwithin{equation}{section}

\def\PointData{%
15/-8,-1/-20,-2/-21,16/-9,15/3,16/-8,%
15/-4,-2/-22,16/-3,17/-8,15/0,16/-4,%
16/-1,15/-2,16/-6,-1/-22,15/-1,17/-5,%
17/-6,4/-18,18/-7,20/-5,19/-5,18/-6,%
19/-6,1/-23,1/-22,6/-16,2/-22,2/-20,%
2/-19,1/-19,20/-2,20/-3,21/-4,21/-3,%
10/15,8/15,9/15,-21/-5,4/-21,5/-19,%
9/16,4/-20,22/-3,4/-19,22/-2,7/-17,%
-20/-5,23/-2,-19/-5,6/-17,6/-19,-21/-2,%
4/-16,6/-15,23/1,24/1,8/17,25/0,%
23/-1,13/-10,4/-15,24/0,5/-15,24/-1,%
11/-13,7/15,5/15,5/20,7/18,6/16,%
1/15,3/17,4/18,5/18,10/-13,5/16,%
14/-11,5/-14,4/-13,11/-11,10/-12,9/-12,%
5/-13,2/16,4/20,5/17,4/21,8/-11,%
9/-11,7/-12,2/21,2/23,5/-12,3/16,%
3/22,0/25,5/-11,7/-11,0/23,6/-10,%
-2/23,-1/23,-22/-2,-23/-3,-1/24,0/20,%
-1/20,-1/21,-2/21,-1/15,-3/22,-3/21,%
-2/16,-2/15,-4/21,-3/19,-4/20,-16/-10,%
6/-11,9/-6,10/-6,7/-9,-16/-9,6/-9,%
4/-9,8/-6,-5/19,6/-8,-4/19,-4/16,%
-4/15,-7/18,-6/19,-6/16,-3/15,-6/18,%
-6/15,5/-8,-8/17,-19/-6,-18/-6,-15/-8,%
-23/-2,-24/-2,-8/15,5/-7,4/-6,-7/15,%
-18/7,8/-4,-17/6,-24/0,-15/10,-5/17,%
-18/-7,4/-7,-17/-5,-14/-8,-24/1,-16/-5,%
-9/16,-25/-1,-23/2,-24/-1,-10/15,-9/15,%
-14/11,-23/1,-14/8,7/-4,-10/13,-11/13,%
-16/-7,-9/13,-15/-6,-11/14,-14/10,-13/8,%
-13/10,-13/9,-16/-6,-14/-6,-14/-4,-8/-11,%
-13/-5,-8/12,-12/-4,-9/11,7/-5,4/-5,%
-9/12,-12/9,-11/10,-9/10,-9/-11,-9/9,%
-11/9,-13/-4,5/-4,5/-5,-10/9,-10/8,%
-11/7,-10/7,-10/6,-10/-7,-11/-7,-11/-5,%
-11/-4,6/-5,-9/6,-10/-3,-6/-9,-10/-4,%
-8/7,-7/-10,-9/-4,-8/-5,-9/-5,-6/-8,%
-7/-5,-8/6,-8/4,-4/8,-5/-8,-5/-7,%
-7/4,-7/-3,-6/4,-7/6,-6/5,-3/7,%
-5/6,-6/-3,-6/-5,-5/5,-2/6,-4/-6,%
-3/5,-4/-7,-6/-4,-5/-4,-4/-5,5/-3,%
-3/6,-4/4,-4/-4,-3/2,-3/4,-2/3,%
-3/3,-4/-2,-3/-2,-1/4,-3/-4,1/3,%
-2/-3,-3/-3,-1/3,0/3,0/2,5/-2,%
4/-2,4/0,3/-1,-1/-4,1/-3,-1/-3,%
0/-3,0/-2,-1/1,-1/-1,0/1,0/-1,%
1/0,2/0,0/0,2/-1,4/-1%
}

\title{A Salem--Spencer-Type Construction for Large Subsets of Integer Grids with No Isosceles Right Triangles}
\author{%
Gyula K\'arolyi\thanks{%
  HUN--REN Alfr\'ed R\'enyi Institute of Mathematics
  Budapest, Hungary; and
  Department of Algebra and Number Theory,
E\"otv\"os University, 
Budapest,  Hungary.
{\it Email address}: {\tt karolyi.gyula@renyi.hu}
}%
\and
J\'ozsef Solymosi\thanks{%
Department of Mathematics, University of British Columbia,
Vancouver, BC, Canada; and
\'Obuda University, Budapest, Hungary.
{\it Email address}: {\tt solymosi@math.ubc.ca}
}%
}
\date{}

\begin{document}
\maketitle

\begin{abstract}
Let $F(n)$ be the largest size of a subset of $\{0,1,\ldots,n-1\}^2$
containing no nondegenerate isosceles right triangle.  We give a modified
Salem--Spencer-type construction over the Gaussian integers showing that
$F(n)=\Omega(n^{1.3})$. The best known upper bound is
$F(n)\ll n^2/(\log n)^{1+c}$ for some absolute constant $c>0$, so there is
still a large gap between the bounds.
\end{abstract}

\section{Introduction}
Finding the largest subset of a $d$-dimensional integer grid avoiding a given
configuration is an important problem. Such questions appear in Ramsey Theory,
Ergodic Theory, Number Theory, Harmonic Analysis, Discrete Geometry and
Combinatorics (see, e.g., \cite{Prendiville, ShkredovSolymosi}). The one-dimensional case, in which the forbidden
configuration is a 3-term arithmetic progression, is the classical problem of
estimating $r_3(N)$, the largest size of a subset of $\{1,\ldots,N\}$
containing no nontrivial 3-term arithmetic progression.
The breakthrough work of Kelley and Meka
\cite{KelleyMeka}, together with the self-contained exposition and subsequent
refinement of Bloom and Sisask
\cite{BloomSisaskExposition,BloomSisaskImprovement}, gave quasipolynomial upper
bounds.  The strongest currently known estimate, due to Raghavan
\cite{Raghavan}, is
\[
  r_3(N)\le
  N\exp\!\left(-c\left(\frac{\log N}{\log\log N}\right)^{1/6}\right)
\]
for some absolute constant $c>0$.
The lower estimate $r_3(N)\ge N\exp(-c'\sqrt{\log N})$, still best known
apart from the value of $c'$, is an old result of Behrend \cite{Behrend}.
A classical construction, showing that there are large subsets of
$\{0,1,\ldots,n-1\}$ without 3-term arithmetical progression is due to Salem
and Spencer \cite{SalemSpencer}.  Their construction introduced a digit-based
way to build very large such sets: one chooses a base and a digit
range for which the relevant linear equation is carry-free, so that a global
configuration reduces to coordinatewise constraints.  We use the same
philosophy for isosceles right triangles (IRT) in the integer lattice,
with Gaussian-integer digits and introduce a finite peeling order in place of
the usual order on integer digits.

Denote by $F(n)$ the maximum size of an IRT-free subset of
$\{0,1,\ldots,n-1\}^2$. Pilatte's extension of the Bloom--Sisask method to
translation-invariant equations with matrix coefficients gives, for some
absolute constant $c>0$,
\[
  F(n)\ll \frac{n^2}{(\log n)^{1+c}};
\]
see \cite{Pilatte}. In particular, $F(n)=o(n^2)$.
From the other direction, even a superlinear lower bound is not obvious.
To establish the promised bound, we identify $(x,y)\in\Z^2$ with $x+iy\in\Zi$.  The construction has four
ingredients.  First, we encode an isosceles right triangle by a linear equation
in $\Zi$.  Second, the digit alphabet is placed in a region where this equation
has no nonzero carries. Finding such an IRT-free alphabet of size greater than
$n$ for some small value of $n$ is enough to obtain
$F(n)=\Omega(n^{\alpha})$ for some
$\alpha>1$ in the spirit of a recursive construction of Szekeres mentioned
in \cite{ErdosTuran}, but we can go further
and allow isosceles right triangles to appear in the alphabet.
Thus, we equip the alphabet with a peeling order
that eliminates the digits one by one.  
This is probably the most novel ingredient of our technique, where 
right-angled vertices play a role analogous to middle terms of 3-term
arithmetic progressions in the construction of Salem and Spencer.
Finally, we restrict the words to a large composition class, which allows
the peeling order to propagate through every digit position.

In the next section we exhibit a general construction, illustrated with the
smallest nontrivial example.
In Section \ref{sec:alphaevolve} we apply it to the AlphaEvolve-assisted
$281$-point certificate.  AlphaEvolve \cite{AlphaEvolve}
is an evolutionary coding agent that uses
programmatic evaluators to guide search; here it serves only
as a discovery mechanism, while the resulting mathematical statement is
checked independently by the certificate in the appendices.

\section{The ingredients of the peeling construction}

Identify the lattice point $(x,y)\in\Z^2$ with the Gaussian integer
$x+iy\in\Zi$.

\begin{observation}
\label{lem:irt-equation}
The lattice points $a,b,c\in\Zi$ form a possibly degenerate
isosceles right triangle
in this clockwise orientation with right angle at vertex $b$
if and only if they satisfy the equation
\begin{equation}
\label{eq:irt}
 a+ic=(1+i)b.
 \end{equation}
Moreover, the only degenerate solutions of \eqref{eq:irt} satisfy $a=b=c$,
corresponding to degenerate triangles.
\end{observation}

\begin{proof}
Clockwise rotation by $\pi/2$ around $b$ moves the point $c$ to $a$
if and only if
\[
  a-b=-i(c-b),
\]
which is a rearrangement of Equation \ref{eq:irt}.
If in a solution to \eqref{eq:irt} any two of $a,b,c$ coincide, the
equation forces all three to coincide.
\end{proof}

The number $\beta\in\Zi$ is a nontrivial element if it is neither 0 nor a
unit in $\Zi$. Equivalently, $|\beta|>1$.  
Let $\beta$ be a nontrivial element of $\Zi$.
A set $P\subset\Zi$ is a partial system of residues modulo $\beta$, if
$a-b\not\in\beta\Zi$ for arbitrary distinct elements $a,b\in P$.
In particular, $|P|\le N(\beta)=|\beta|^2$. We refer to such a set as an
alphabet, whose elements can be used as digits for unique expansion
of Gaussian integers in base $\beta$ under some mild conditions.

\begin{definition}
\label{def:carry-free}
Let $\beta\in\Zi$ be a nontrivial element and let $P\subset\Zi$ be
a partial system of residues modulo $\beta$.  We say that the alphabet
$P$ is \emph{IRT-carry-free in base $\beta$} if 
\begin{equation}
\label{eq:carry}
a+ic-(1+i)b\in\beta\Zi,\quad a,b,c\in P \implies a+ic=(1+i)b.
\end{equation}
\end{definition}

\noindent
For example, the partial system of residues $P=\{0,1,2\}$ modulo $\beta$
is not an IRT-carry-free alphabet either in base $\beta=2+i$
or in base $\beta=2+2i$, for in the first case $a=0$, $b=2$, $c=1$
is a solution of Equation \ref{eq:carry} modulo $\beta$,
but is not the solution of the equation itself, and the same holds
in the second case with the degenerate triple $a=c=0$, $b=2$.
On the other hand, $P=\{0,1,i\}$ is an IRT-carry-free alphabet modulo $2+2i$, see
Example \ref{ex:1stnontrivial} below.
To prove that our constructions are IRT-free, we will use 
Condition \ref{eq:carry} to check Equation \ref{eq:irt} digitwise.

For an ordering $p_1,\ldots,p_q$ of the elements of an alphabet $P$, put
\[
  \cR_t=\{p_t,p_{t+1},\ldots,p_q\}.
\]

\begin{definition}
\label{def:peeling}
The ordering $p_1,\ldots,p_q$
is an \emph{IRT-peeling order} of $P$ if, for every
$t\in\{1,\ldots,q\}$, the only solution of
\begin{equation}
  \label{eq:peeling}
  a+ic=(1+i)p_t,
  \qquad a,c\in\cR_t,
 \end{equation}
is $a=c=p_t$. Such an ordering is also referred to as a peeling sequence.
\end{definition}

\noindent
Thus, there is no nondegenerate isosceles triangle in
the remaining suffix $\cR_t$ with a right-angle at vertex $p_t$.
The whole alphabet need not itself be IRT-free.

\begin{example}
\label{ex:1stnontrivial}
{\rm Let $\beta=2+2i$. The set $P=\{0,1,i\}$ is a partial system of
  residues modulo $\beta$, since $|a-b|\le \sqrt{2}<|\beta|$ for any
  $a,b\in P$. The three digits of this alphabet form an
  isosceles right triangle,
  the only nondegenerate solution of \eqref{eq:irt} being $a=1$, $b=0$, $c=i$.
The digit $1$ is not a right-angle vertex of any triangle in the full set, so it
can be peeled first. Once $1$ is removed, triangle $abc$ disappears,
so $0$ can be peeled next, followed by $i$. Thus $1,0,i$
is an IRT-peeling order of $P$; in fact, it can be equipped with
four different IRT-peeling orders.
The alphabet $P$ is also IRT-carry-free.
To verify Condition \ref{eq:carry}, let $a,b,c\in\{0,1,i\}$ and
$\delta=a+ic-(1+i)b$. It is enough to check that $|\delta|<|\beta|=2\sqrt{2}$.
If one of $a,b,c$ is equal to 0, then the triangle inequality yields
$|\delta|\le 1+\sqrt{2}$.
Otherwise two of them coincide and $|\delta|\le 2$. As we shall see right after the next theorem, this simple example gives a non-trivial bound already.}
\end{example}

\begin{remark}
 \label{rem:squarefree}
 {\rm No finite set that contains the vertices of a square can be equipped
    with an IRT-peeling order. The 8-point configuration depicted in
    Figure \ref{fig:squarefree} must also be excluded for the same reason:
    Each point is a vertex of an IRT within the configuration. There are arbitrarily large
    point sets which are not peelable, but after removing any vertex, the set becomes peelable.

\begin{figure}[h]
\centering
\begin{tikzpicture}[x=1cm,y=1cm]
  \draw[gray!70] (0,0) grid (2,3);

  \foreach \x/\y in {
    0/0, 1/0,
    0/1, 2/1,
    0/2, 2/2,
    0/3, 1/3}
    \fill (\x,\y) circle (2.2pt);
\end{tikzpicture}
\caption{A square-free forbidden configuration}
\label{fig:squarefree}
\end{figure}
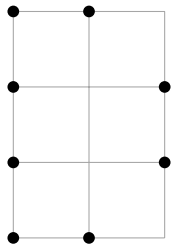
}

\end{remark}

\begin{theorem}
\label{thm:general-construction}
Let $\beta\in\Zi$ be a nontrivial element, and let $P$ be an
IRT-carry-free alphabet in base $\beta$ with an IRT-peeling order
$p_1,\ldots,p_q$.  For every integer $m\ge 1$, there is a set
$A_m\subset\Zi$ contained in a circle of radius $\ll \beta^{m}$,
which does not contain any nondegenerate isosceles right triangle and satisfies
\[
\abs{A_m}\gg \frac{q^m}{m^{(q-1)/2}}.
\]
Consequently,
\begin{equation}
  F(n)\gg
  \frac{n^{\log q/\log\abs\beta}}{(\log n)^{(q-1)/2}}.
  \label{eq:general-asymptotic}
\end{equation}
The constants implied in the Vinogradov symbols depend only on $\beta$ and $P$.
\end{theorem}

\begin{proof}
For a word $w=(w_0,\ldots,w_{m-1})\in P^m$, define
\[
  \Phi(w)=\sum_{j=0}^{m-1}w_j\beta^j \in \Zi.
\]
The first condition in Definition~\ref{def:carry-free}, followed inductively
from the least significant digit upward, shows that $\Phi$ is injective.

The composition of $w$ is the vector
\[
  \nu(w)=(\nu_1(w),\ldots,\nu_q(w)),
\]
where $\nu_t(w)$ counts the coordinates equal to $p_t$:
$\nu_t(w)=|N_t(w)|$, where $N_t(w)\subseteq [0,m-1]$
denotes the set of positions the digit $p_t$ occupies in $w$.
There are exactly $\binom{m+q-1}{q-1}$ possible compositions.
Hence some composition class $W_m\subset P^m$ satisfies
\[
  \abs{W_m}\ge \frac{q^m}{\binom{m+q-1}{q-1}}.
  \]
Using Stirling's approximation formula to estimate the size of
a most popular composition class, belonging to a composition with nearly
equal coordinates, yields the refinement
\[
  \abs{W_m}\gg \frac{q^m}{m^{(q-1)/2}}.
  \]
Set $A_m=\Phi(W_m)$.

If $A_m$ contains the vertex set of an isosceles right triangle, then
in view of Observation \ref{lem:irt-equation}, Equation \ref{eq:irt}
has a nondegenerate solution in $A_m$.
Suppose that $A+iC=(1+i)B$ holds for some $A,B,C\in A_m$. We claim that $A=B=C$.
The words $a=(a_j)$, $b=(b_j)$, $(c=c_j)$ corresponding to $A$, $B$, and $C$,
respectively, satisfy
\[
  \sum_{j=0}^{m-1}
  \bigl(a_j+ic_j-(1+i)b_j\bigr)\beta^j=0.
\]
Reduction modulo $\beta$ and Condition \ref{eq:carry} give
$a_0+ic_0=(1+i)b_0$. Dividing by $\beta$ and iterating yields
\begin{equation}
  a_j+ic_j=(1+i)b_j
  \qquad (0\le j<m).
  \label{eq:digitwise}
\end{equation}

To prove that the three words coincide coordinatewise
we apply complete induction for $t=1,\dots,q$ by peeling off
the digits $p_i$ one by one as follows.
Assume that for some $1\le t <q$, $N_i(a)=N_i(b)=N_i(c)=N_i$ holds for every
$i<t$. We show that the assumption, void if $t=1$, implies
$N_t(a)=N_t(b)=N_t(c)$.
Let $w\in \{a,b,c\}$. If $w_j=p_t$, then
\[
j\in \{0,1,\dots,m-1\} \setminus (N_1\cup \dots \cup N_{t-1}).
\]
Whenever $b_j=p_t$, Equations \ref{eq:digitwise} and
\ref{eq:peeling} force $a_j=c_j=p_t$. Thus, 
$\nu_t(a)=\nu_t(b)=\nu_t(c)$ implies $N_t(a)=N_t(b)=N_t(c)$.
The inductive step thus completed, the claim follows by taking $t=q$.

Finally, let $M=\max_{p\in P}\abs p$. Every element of $A_m$ has modulus at
most
\[
  M\sum_{j=0}^{m-1}\abs\beta^j
  =M\frac{\abs\beta^m-1}{\abs\beta-1}.
\]
$A_m$ therefore lies in an $N\times N$ integer grid, where
$N\le K\abs\beta^m$ for a constant $K$ depending only on $\beta$ and $P$.
A suitable translation of $A_m$ is an IRT-free subset of 
$\{0,1,\ldots,n-1\}^2$ for any $n\ge N$.
For any given sufficiently large $n$, choose the largest $m$ with
$K\abs\beta^m\le n$. Then $q^m\gg n^{\log q/\log\abs\beta}$ and
$m^{(q-1)/2}\ll (\log n)^{(q-1)/2}$.
This proves \eqref{eq:general-asymptotic}.
\end{proof}

\noindent
Applied to the construction in Example \ref{ex:1stnontrivial},
Theorem~\ref{thm:general-construction} gives
\[
  F(n)\gg \frac{n^{\alpha}}{\log n},
  \quad \text{where} \quad
   \alpha=\frac{\log3}{\log(2\sqrt2)}
        =\frac{2\log3}{3\log2}
        \approx1.05664.
\]

It is not difficult to verify that this is the smallest meaningful
construction
that can produce an exponent $\alpha$ strictly larger than $1$ through
Theorem~\ref{thm:general-construction}.

\section{A specific scheme}
\label{sec:specscheme}

In view of Theorem \ref{eq:general-asymptotic}, the goal is to optimize
$\log q/\log\abs\beta$ for a peelable IRT-carry-free alphabet $P$ of size $q$
in a suitable Gaussian base $\beta$. To ensure that $P$ is indeed a partial
system of residues modulo $\beta$, we locate $P$ in a fundamental domain
$\Omega$ of the lattice $\beta\Zi$. 
To avoid carries, we study 
well-shaped regions $\cC$ inside $\Omega$ that satisfy 
the geometric condition that the Minkowski-sum of
$\cC$ and $i\cC$ is disjoint from $(1+i)\cC+\beta\Zi$. Thus, every subset
of such a carry-free region can serve as an IRT-carry-free alphabet. 

\begin{example}
\label{ex:2ndnontrivial}
{\rm Let $\beta=6+5i$, then $N(\beta)=61$.
  Consider the square with vertices $O(0,0)$, $A(6,5)$, $B(1,11)$, $C(-5,6)$,
  its midpoint is $K(1/2,11/2)$. The 60 gridpoints inside the square, together
  with $A$, form a fundamental domain $\Omega$ for $\beta\Zi$, the 
  15 gridpoints inside triangle $OAK$, together with $A$, form a carry-free
  region $\cC\subseteq \Omega$. One readily sees that
  \[
(1,1),\ (1,2),\ (2,2),\ (2,3),\ (3,3),\ (3,4),\ (4,4),\ (4,5),\ (5,5),\ (6,5)
  \]
  is a peeling sequence of length $q=10$ in $\cC$.
An application of Theorem~\ref{thm:general-construction} gives a bound
\begin{equation}
  \label{eq:1.12024}
  F(n)\gg \frac{n^{\alpha}}{(\log n)^{4.5}},
  \quad \text{where} \quad
   \alpha=\frac{\log q}{\log\abs\beta}
        =\frac{\log100}{\log61}
        \approx1.12024.
\end{equation}
}
\end{example}
    
\begin{remark}
  \label{rem:limit}
    {\rm Such a staircase-like IRT-carry-free peeling  sequence of length
      $q=2k$ in base $\beta=(k+1)+ki$ can be constructed for any positive
      integer $k$. Since the function ${\log(4k^2)}/{\log(2k^2+2k+1})$ attains
      its maximum value for positive integers at $k=5$, they cannot be used
      to obtain a better lower bound.
  }
\end{remark}

Experimenting with different bases and IRT-carry-free regions we found that
the most likely candidates to obtain a good bound are those Gaussian integers
$\beta$ in which the real and imaginary parts are nearly equal, together
with a carry-free region that satisfies $\cC\approx \Omega/2$ when the
fundamental domain $\Omega$ is located centrally.
To capture the condition
\[
(\cC +i\cC)\cap \left((1+i)\cC+\beta\Zi\right) = \emptyset
\]
algebraically, introduce the diagonal coordinates
\[
  \U(z)=x+y,
  \qquad
  \V(z)=x-y.
\]
for $z=x+iy\in\Zi$. For each positive integer $k$, consider a Gaussian
integer $\beta$ of the form
\[
\beta=(2k+1)+(2k+1)i.
\]
Then
\begin{equation}
  z\in\beta\Zi
  \quad\Longleftrightarrow\quad
  \U(z)\equiv0\pmod{4k+2}
  \ \text{and}\ 
  \V(z)\equiv0\pmod{4k+2}.
  \label{eq:beta-divisibility}
\end{equation}
A centrally located fundamental domain for $\Zi/(\beta)$ is
\[
  \Omega=
  \{z\in\Zi:-2k-1\le\U(z)\le 2k,\ -2k-1 \le\V(z)\le 2k\}.
\]

Inside $\Omega$ consider the smaller diamond
\begin{equation}
  \cC=
  \{z\in\Zi:-k-1 \le\U(z)\le k,\ -k\le\V(z)\le k\}.
  \label{eq:carry-free-diamond}
\end{equation}

\begin{lemma}
\label{lem:carry-free}  
The region $\cC$ is IRT-carry-free in base $\beta$.
\end{lemma}

\begin{proof}
To check Condition \ref{eq:carry},
consider $\delta=a+ic-(1+i)b$ for some $a,b,c\in \cC$.
The identities
\[
  \U(iz)=\V(z),\qquad \V(iz)=-\U(z),
\]
and
\[
  \U((1+i)z)=\U(z)+\V(z),
  \qquad
  \V((1+i)z)=\V(z)-\U(z)
\]
give
\begin{align*}
  \U(\delta)&=\U(a)+\V(c)-\U(b)-\V(b),\\
  \V(\delta)&=\V(a)-\U(c)-\V(b)+\U(b).
\end{align*}
The bounds in \eqref{eq:carry-free-diamond} imply
\[
  -4k-1\le\U(\delta)\le 4k+1,
  \qquad
  -4k-1\le\V(\delta)\le 4k+1.
\]
If $\delta\in\beta\Zi$, Equation \eqref{eq:beta-divisibility} forces both
coordinates to be divisible by $4k+2$, hence both vanish. This
implies $\delta=0$.
\end{proof}

\section{The AlphaEvolve-optimized 281-point construction}
\label{sec:alphaevolve}

To find a good construction, we used AlphaEvolve, an AI coding agent developed
by Google DeepMind that uses large language models and evolutionary
computation to automatically design, evaluate, and optimize complex algorithms. 
Increasing the modulus of $\beta$, first the points of $P$ tend to be
concentrated next to the boundary of its convex hull in a zig-zagging
fashion, like in Example \ref{ex:2ndnontrivial}, then points start to
appear inside the the convex hull with increasing density.

Set $\beta=51+51i$. According to Lemma \ref{lem:carry-free}, the region
\[
\cC=  \{z\in\Zi:-26 \le\U(z)\le 25,\ -25\le\V(z)\le 25\}
\]
is IRT-carry-free in base $\beta$. 
Appendix~\ref{app:points} gives an ordered list
$p_1,\ldots,p_{281}$ of points in $\cC$.
The list was obtained by an AlphaEvolve-assisted search and is plotted in
Figure~\ref{fig:points}. The order, rather than only the underlying point
set $P=\{p_1,\ldots,p_{281}\}$
is the finite certificate needed by the peeling argument.

\begin{figure}[H]
\centering
\begin{tikzpicture}[x=0.105cm,y=0.105cm]
  \draw[->] (-28,0) -- (28,0) node[right] {$x$};
  \draw[->] (0,-28) -- (0,28) node[above] {$y$};
  \draw[thick]
    (-25.5,-0.5) -- (-0.5,-25.5) -- (25,0) -- (0,25) -- cycle;
  \foreach \x/\y in \PointData {
    \fill (\x,\y) circle[radius=0.55pt];
  }
  \node[above right] at (17,16) {$\cC$};
\end{tikzpicture}
\caption{The $281$ Gaussian digits in the carry-free diamond $\cC$.
}
\label{fig:points}
\end{figure}
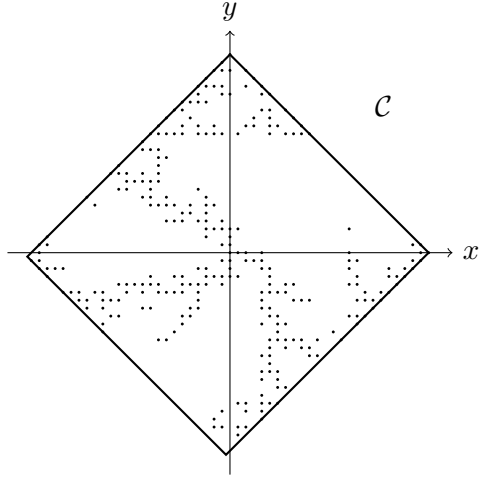

\begin{proposition}
\label{prop:certificate}
The ordered list in Appendix~\ref{app:points} consists of $281$ distinct points
of $\cC$ and is an IRT-peeling order.  Its canonical SHA-256 digest is
\begin{center}
\ttfamily\small
9ebeedde004c4d30a1da52f5f8fae4f148fee3d524f540c7e35369f192e83724.
\end{center}
\end{proposition}

\begin{proof}
This is a finite deterministic check. For each $t$, form
$\cR_t=\{p_t,\ldots,p_{281}\}$. For each $c\in\cR_t$, Equation
\eqref{eq:peeling} determines the only possible endpoint
\[
  a=(1+i)p_t-ic.
\]
The verifier checks that $a\notin\cR_t$ unless $a=c=p_t$. It also checks the
point count, distinctness, membership in $\cC$, distinct residues modulo
$\beta$, and every possible carry. The complete verifier and its
expected output appear in Appendix~\ref{app:verification}.
\end{proof}

Theorem \ref{thm:general-construction} applied to this construction
yields or main result.

\begin{theorem}
\label{thm:main}
For every integer $n>1$,
\[
  F(n)\gg \frac{n^\alpha}{(\log n)^{140}},
  \quad \text{where} \quad
  \alpha=\frac{\log281}{\log\abs{51+51i}}
        =\frac{2\log281}{\log5202}
        \approx1.317865485534.
\]
In particular, $F(n)\ge n^{\alpha-o(1)}$.
\end{theorem}

\begin{remark}
{\rm The adjective ``optimized'' refers to the computational search objective.
  The point list and verifier certify the feasibility of the $281$-digit
  construction
and hence Theorem~\ref{thm:main}. They do not constitute an upper-bound
certificate proving that $281$ is the maximum possible alphabet size in
$\cC$, nor that this base is globally optimal among all Gaussian bases.}
\end{remark}

\section{Concluding remarks}

The proof separates a geometric argument from a finite search. The geometric
part supplies a carry-free residue region for a Gaussian base; the finite part
finds a large peelable alphabet inside that region. This separation makes the
result independently checkable and suggests several extensions: different
Gaussian bases, asymmetric carry-free regions, block digits, and exact
optimization methods capable of producing upper-bound certificates in addition
to large feasible alphabets.

\bigskip
\noindent
{\bf Acknowledgements.}
J.S. was supported by an NSERC Discovery Grant and by the Excellence 151341 
grant by the National Research, Development, and Innovation Office of Hungary.
We thank the Google DeepMind team for providing access to AlphaEvolve, which was used to assist in algorithm discovery and optimization for this work.

\bigskip
\noindent
{\bf Declaration.}
There are no relevant financial or non-financial competing interests to report.

\appendix

\section{The ordered 281-point certificate}
\label{app:points}

The $t$th item below is the Gaussian digit $p_t=x+iy$ represented by the pair
$(x,y)$. This order is part of the certificate.

\begingroup
\scriptsize
\setlength{\columnsep}{1.2em}
\begin{multicols}{4}
\begin{enumerate}[leftmargin=*,label=\arabic*.,itemsep=0pt,parsep=0pt,topsep=0pt]
\foreach \x/\y in \PointData {\item $(\x,\y)$}
\end{enumerate}
\end{multicols}
\endgroup

\section{Standalone verification}
\label{app:verification}

The following Python program reads the machine-readable \verb|\PointData| macro
directly from this TeX source. Save it as
\verb|verify_alphaevolve_281.py| next to the source and run
\begin{center}
\verb|python verify_alphaevolve_281.py irt_peeling_paper.tex|.
\end{center}
The canonical digest is computed from the space-separated sequence
\verb|(x,y)| with no spaces inside pairs and with one terminal newline.

\begin{lstlisting}
#!/usr/bin/env python3
"""Verify the 281-point Gaussian-digit certificate embedded in the TeX source."""

from __future__ import annotations

import hashlib
import math
import re
import sys
from pathlib import Path

EXPECTED_COUNT = 281
EXPECTED_SHA256 = "9ebeedde004c4d30a1da52f5f8fae4f148fee3d524f540c7e35369f192e83724"

Point = tuple[int, int]


def load_points(tex_path: Path) -> list[Point]:
    text = tex_path.read_text(encoding="utf-8")
    match = re.search(r"\\def\\PointData\{(.*?)\}", text, flags=re.S)
    if match is None:
        raise ValueError("Could not find \\def\\PointData{...} in the TeX source")
    return [
        (int(x), int(y))
        for x, y in re.findall(r"(-?\d+)\s*/\s*(-?\d+)", match.group(1))
    ]


def add(z: Point, w: Point) -> Point:
    return z[0] + w[0], z[1] + w[1]


def sub(z: Point, w: Point) -> Point:
    return z[0] - w[0], z[1] - w[1]


def mul_i(z: Point) -> Point:
    return -z[1], z[0]


def mul_1_plus_i(z: Point) -> Point:
    return z[0] - z[1], z[0] + z[1]


def diagonal(z: Point) -> Point:
    x, y = z
    return x + y, x - y


def residue_mod_beta(z: Point) -> Point:
    # beta = 51+51i. A residue is determined by U(z),V(z) modulo 102.
    u, v = diagonal(z)
    return u % 102, v % 102


def canonical_digest(points: list[Point]) -> str:
    payload = " ".join(f"({x},{y})" for x, y in points) + "\n"
    return hashlib.sha256(payload.encode("ascii")).hexdigest()


def verify(points: list[Point]) -> None:
    assert len(points) == EXPECTED_COUNT
    assert len(set(points)) == EXPECTED_COUNT
    assert canonical_digest(points) == EXPECTED_SHA256

    # Membership in the carry-free diamond C.
    for z in points:
        u, v = diagonal(z)
        assert -26 <= u <= 25
        assert -25 <= v <= 25

    # Distinct digit residues modulo beta.
    by_residue = {residue_mod_beta(z): z for z in points}
    assert len(by_residue) == EXPECTED_COUNT

    # Direct O(q^2) check of all possible one-digit carries.
    # For fixed b,c, a carry can occur only for the unique digit a having
    # the same beta-residue as (1+i)b-ic, if such a digit exists.
    for b in points:
        for c in points:
            target = sub(mul_1_plus_i(b), mul_i(c))
            a = by_residue.get(residue_mod_beta(target))
            if a is not None:
                delta = sub(add(a, mul_i(c)), mul_1_plus_i(b))
                assert delta == (0, 0)

    # Suffix-peeling certificate.
    for t, p in enumerate(points):
        remaining = set(points[t:])
        rhs = mul_1_plus_i(p)
        for c in remaining:
            a = sub(rhs, mul_i(c))
            assert a not in remaining or (a == p and c == p)


def main() -> None:
    default_tex = Path(__file__).with_name("irt_peeling_paper.tex")
    tex_path = Path(sys.argv[1]) if len(sys.argv) > 1 else default_tex
    points = load_points(tex_path)
    verify(points)
    rho = math.log(len(points)) / math.log(abs(complex(51, 51)))
    print(f"points: {len(points)}")
    print(f"sha256: {canonical_digest(points)}")
    print("membership: verified")
    print("digit residues: verified")
    print("one-digit carry check: verified")
    print("peeling order: verified")
    print(f"rho: {rho:.15f}")


if __name__ == "__main__":
    main()
\end{lstlisting}

The expected output is
\begin{lstlisting}[language={}]
points: 281
sha256: 9ebeedde004c4d30a1da52f5f8fae4f148fee3d524f540c7e35369f192e83724
membership: verified
digit residues: verified
one-digit carry check: verified
peeling order: verified
rho: 1.317865485534210
\end{lstlisting}

\end{document}